\documentclass[12pt]{article} 

\usepackage{amsmath}
\usepackage{amsthm}
\usepackage{amsfonts}
\usepackage{mathrsfs}
\usepackage{setspace}
\usepackage{fullpage}
\usepackage{amssymb}
\usepackage{enumitem}
\usepackage{comment}
\usepackage{pgf,tikz}
\usepackage{graphicx}
\usepackage{xcolor}
\usepackage{caption}
\usepackage{mathrsfs}
\usepackage{mathrsfs}
\usepackage{enumitem}
\usepackage{mathtools}

\usepackage[colorlinks,citecolor=blue,bookmarks=true,urlcolor=black]{hyperref}

\usepackage[maxbibnames=99]{biblatex}
\numberwithin{equation}{section}

\theoremstyle{plain}
\newtheorem{theorem}{Theorem}[section]
\newtheorem{lemma}[theorem]{Lemma}
\newtheorem{proposition}[theorem]{Proposition}
\newtheorem{corollary}[theorem]{Corollary}
\newtheorem{claim}[theorem]{Claim}
\newtheorem{observation}[theorem]{Observation}
\newtheorem{conjecture}[theorem]{Conjecture}
\newtheorem{definition}[theorem]{Definition}

\makeatletter
\renewenvironment{proof}[1][\proofname]
{\par\pushQED{\qed}
	\normalfont\topsep6\p@\@plus6\p@\relax\trivlist
	\item[\hskip\labelsep\bfseries#1\@addpunct{.}]
	\ignorespaces}
{\popQED\endtrivlist\@endpefalse}
\makeatother

\newcommand{\Pro}{\mathcal{P}}
\newcommand{\smkl}{S^-_{k,\ell}}
\newcommand{\smtwol}{S^-_{2,\ell}}

\DeclarePairedDelimiter\card{\lvert}{\rvert}

\title{Unavoidable subgraphs in digraphs with large out-degrees}

\author{
\setlength{\tabcolsep}{20pt}

\begin{tabular}{ccc}
Tom\'{a}\v{s} Hons\thanks{Computer Science Institute (I\'UUK, MFF), Charles University, Prague, Czech Republic.} & Tereza Klimo\v{s}ov\'{a}\thanks{Department of Applied Mathematics (KAM, MFF), Charles University, Prague, Czech Republic. \\Emails: \{\href{mailto:honst@iuuk.mff.cuni.cz}{honst}, \href{mailto:miksanik@iuuk.mff.cuni.cz}{miksanik}, \href{mailto:josef.tkadlec@iuuk.mff.cuni.cz}{josef.tkadlec}\}@iuuk.mff.cuni.cz; \{\href{mailto:gaurav@kam.mff.cuni.cz}{gaurav}, \href{mailto:tereza@kam.mff.cuni.cz}{tereza}, \href{mailto:tyomkyn@kam.mff.cuni.cz}{tyomkyn}\}@kam.mff.cuni.cz. 
TH and MT have been supported by GA\v{C}R
grant 25-17377S and ERC Synergy Grant DYNASNET 810115.
TK has been supported by GAČR grant 25-16847S and Charles Univ.\ project UNCE 24/SCI/008.
GK has been supported by GA\v{C}R
grant 25-17377S and Charles Univ.\ project UNCE 24/SCI/008.
DM has been supported by the ERC-CZ project LL2328.
JT has been supported by GA\v{C}R
grant 25-17377S and Charles Univ.\ projects UNCE 24/SCI/008 and PRIMUS 24/SCI/012.} & Gaurav Kucheriya\footnotemark[2] \\ David Mik\v{s}an\'ik\footnotemark[1] & Josef Tkadlec\footnotemark[1] & Mykhaylo Tyomkyn\footnotemark[2]
\end{tabular}}

\date{}
\begin{document}
\maketitle
\begin{abstract}
We ask the question, which oriented trees $T$ must be contained as subgraphs in every finite directed graph of sufficiently large minimum out-degree. We formulate the following simple condition: all vertices in $T$ of in-degree at least $2$ must be on the same `level' in the natural height function of $T$. We prove this condition to be necessary and conjecture it to be sufficient. In support of our conjecture, we prove it for a fairly general class of trees. 

An essential tool in the latter proof, and a question interesting in its own right, is finding large subdivided in-stars in a directed graph of large minimum out-degree. We conjecture that any digraph and oriented graph of minimum out-degree at least $k\ell$ and $k\ell/2$, respectively, contains the $(k-1)$-subdivision of the in-star with $\ell$ leaves as a subgraph; this would be tight and 
generalizes a conjecture of Thomass\'e. We prove this for digraphs and $k=2$ up to a factor of less than $2$.
\end{abstract}

\newpage

\section{Introduction}\label{sec:intro}

One of the main focuses in the study of finite\footnote{All directed graphs considered in this paper will be finite. For better readability we will keep it implicit.} directed graphs has been the investigation of properties of digraphs and oriented graphs of large minimum out-degree. For instance, the famous Caccetta-H\"aggkvist conjecture~\cite{caccetta1978minimal} states that every digraph $G$ of order $n$ with minimum out-degree $\delta^+(G)\geq d$ has a directed cycle of length at most $\lceil n/d \rceil$. A conjecture of Thomass\'e  (see~\cite{bang2008digraphs, sullivan2006summary}) claims that any oriented graph $G$ with $\delta^+(G)\geq d$ contains a directed path of length $2d$, and recently Cheng and Keevash~\cite{cheng2024length}, proved a lower bound of $3d/2$. The Bermond-Thomassen conjecture~\cite{bermond1981cycles} states that every digraph $G$ with $\delta^+(G)\geq 2d-1$ contains $d$ disjoint directed cycles. Alon~\cite{alon1996disjoint} and, more recently, Buci\'c~\cite{bucic2018improved} proved it with $2d-1$ replaced by $64d$ and $18d$, respectively. 

As quantitative bounds in many of these problems are hard to obtain or sometimes even to guess, it is natural to ask for a property of digraphs if it holds in all digraphs of \emph{sufficiently large} minimum out-degree. For example, Stiebitz~\cite{stiebitz1996decomposing} and, independently, Alon~\cite{alon2006splitting} asked whether for every $a, b \geq 1$ there exists $F(a, b)$ such that $\delta^+(G)\geq F(a,b)$ implies that $V(G)$ can be partitioned
into two non-empty parts $A$ and $B$ with $\delta^+(G[A])\geq a$ and $\delta^+(G[B])\geq b$. In a recent breakthrough, Christoph, Petrova and Steiner~\cite{christoph2023note} reduced this question to that of existence of $F(2,2)$.  

Mader~\cite{mader1995exixtence} conjectured the existence of a function $f$ so that $\delta^+(G)\geq f(k)$ implies that $G$ contains a subdivision of the transitive tournament of order $k$, and proved it for $k \leq 4$~\cite{Mader1996}. This sparked an interest in finding subdivisions of fixed digraphs in digraphs of large out-degree. Aboulker, Cohen, Havet, Lochet, Moura and Thomass\'e~\cite{aboulker1610subdivisions} defined a digraph $H$ to be \emph {$\delta^+$-maderian} if, for some value $d$, every $G$ with $\delta^+(G)\geq d$ contains a subdivision of $H$. In this terminology, Mader's conjecture states that every acyclic digraph is $\delta^+$-maderian. In support of this, the authors of~\cite{aboulker1610subdivisions} proved among other results that every in-arborescence (i.e. a tree with all edges oriented towards a designated
root vertex) is $\delta^+$-maderian. They also conjectured that every orientation of a cycle is $\delta^+$-maderian, and this was recently confirmed by Gishboliner, Steiner and Szab\'o~\cite{gishboliner2022oriented}.

In this paper we are asking which digraphs $H$ must be contained in all digraphs of sufficiently large minimum out-degree \emph{as subgraphs}. To our surprise, we were not able to find any previous systematic study of the topic, despite the question being natural. Note that orientations of each graph with a cycle can be avoided by taking a $2d$-regular connected unoriented graph of large girth, and orienting its edges via an Euler circuit. Hence, we may assume that $H$ is an orientation of a tree (or a forest, which again reduces to trees). 

\begin{definition}
An oriented tree $T$ is $\delta^+$-enforcible if there exists $d=d(T)$ such that every digraph $G$ with $\delta^+(G)\geq d$ contains $T$ as a subgraph.     
\end{definition}

\noindent
A simple greedy embedding certifies that every out-arborescence is $\delta^+$-enforcible. Much less obviously, by the aforementioned result of Aboulker et al.~\cite{aboulker1610subdivisions} on in-arborescences, every subdivision of the in-star is $\delta^+$-enforcible. We remark that it is \emph{not} true that every in-arborescence is $\delta^+$-enforcible, as will follow from our Theorem~\ref{thm:level}. Another known family of $\delta^+$-enforcible trees are the antidirected trees (that is, trees containing no directed path of length $2$). By a theorem of Burr~\cite{Burr_1982}, $\delta^+(G)\geq 4k$ implies that $G$ contains every antidirected $k$-edge tree as a subgraph.  

For an oriented tree $T$, its \emph{height function} $h_T\colon V(T)\rightarrow \mathbb{Z}$ is a function satisfying $h_T(v)=h_T(u)+1$ for every edge $(u,v)\in E(T)$. It is clear that the height function is well-defined and unique up to an additive constant. Since we will only care about the relative values of $h_T$ between the vertices of $T$, we will, slightly abusing the notion, speak of ``the height function $h_T$.''

\begin{definition}\label{def:grounded}
An oriented tree $T$ is \emph{grounded} if $h_T(v)$ is constant for all vertices $v$ of in-degree at least $2$.
\end{definition}

\noindent
Note that all the above examples of $\delta^+$-enforcible trees are grounded. We show that they have to be, and conjecture that this is an `if and only if' relationship.

\begin{theorem}\label{thm:level}
Every $\delta^+$-enforcible tree is grounded.
\end{theorem}
\begin{conjecture}[KAMAK tree conjecture\footnote{Named after the KAMAK 2024 workshop where the present work was initiated.}]\label{conj:KAMAK-tree}
Every grounded tree is $\delta^+$-enforcible. Hence, an oriented tree is $\delta^+$-enforcible if and only if it is grounded.     
\end{conjecture}
\noindent
In support of Conjecture~\ref{conj:KAMAK-tree}, we prove it for a fairly general class of trees. 
\noindent

\begin{figure}[htbp]
  \centering
  \includegraphics[scale=0.85]{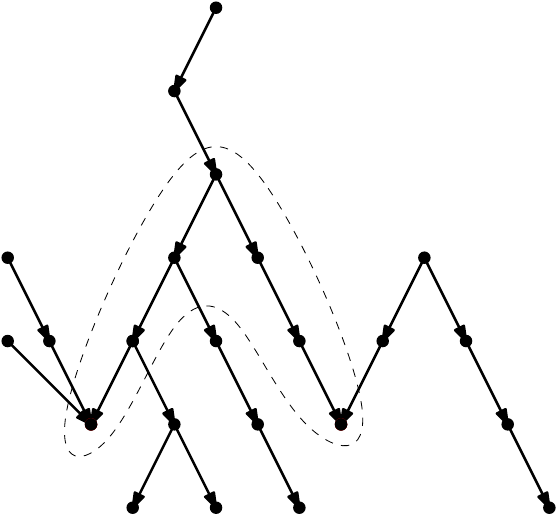}
  \caption{An example of a tree $T$ that is $\delta^+$-enforcible by Theorem~\ref{thm:groundedclass}, highlighting the minimal subtree containing the set $U(T)$.}
\end{figure}

\begin{theorem}\label{thm:groundedclass}
Suppose that $T$ is grounded and that the minimal subtree of $T$ containing the vertex set $U(T)=\{v\in V(T): \deg^-(v)\geq 2\}$ is an out-arborescence. Then $T$ is $\delta^+$-enforcible.  
\end{theorem}

\begin{corollary}\label{cor:twosinks}
Every grounded tree with at most two vertices of in-degree at least $2$ is $\delta^+$-enforcible. 
\end{corollary}

\noindent
We remark that there has been a considerable body of past and recent work (\cite{jackson1981long, kathapurkar2022spanning, KLIMOSOVA2023113515, penev2025twoblockpathsorientedgraphs, skokan2024alternating}, see also~\cite{Stein_2024} for a survey) on finding oriented trees in digraphs and oriented graphs of large minimum \emph{semi-degree} $\delta^0(G)$, which is the minimum of all in- and out-degrees in $G$. The focus there is different, namely on explicit bounds on $\delta^0(G)$, as any fixed oriented tree can be greedily embedded into 
a digraph of sufficiently large minimum semi-degree. That being said, our last result deals with explicit bounds in the minimum out-degree setting. 

Let $\smkl$ be the $(k-1)$-subdivision of the in-star with $\ell$ leaves. An essential tool in our proof of Theorem~\ref{thm:groundedclass} is the existence of $\smkl$ in digraphs of large minimum out-degree $d$, established in~\cite{aboulker1610subdivisions}; an inspection of the proof shows that (in the proof) $d$ needs to be at least $\ell^{k!}$. In order to achieve better, perhaps even tight, quantitative bounds in Theorem~\ref{thm:groundedclass} and towards Conjecture~\ref{conj:KAMAK-tree}, it would be desirable to find tight out-degree bounds for the containment of $\smkl$. In this regard, we make the following conjecture.

\begin{conjecture}[Giant spider conjecture]
For every $k\geq 2$ and $\ell\geq 1$
\begin{enumerate}
    \item[(i)] any digraph $G$ with $\delta^+(G)\geq k\ell$ contains $\smkl$ as a subgraph.
    \item[(ii)] any oriented graph $G$ with $\delta^+(G)\geq k\ell/2$ contains $\smkl$ as a subgraph. 
\end{enumerate}
\end{conjecture} 

\noindent
These bounds would be tight, by the examples of the complete digraph and a regular tournament of order $k\ell$, respectively. 
Note that for $\ell=1$, when $\smkl$ is the path of length $k$, the first statement is the obvious greedy algorithm bound, while the second is Thomass\'e's conjecture~\cite{bang2008digraphs, sullivan2006summary}, mentioned earlier. Addressing the other `extreme' case $k=2$, we prove a linear bound.
\begin{theorem}\label{thm:linear_spiders}
For every $\ell\geq 1$, any digraph $G$ with $\delta^+(G) \geq (1+\sqrt{5})\ell \approx3.23 \cdot \ell$ contains $S^-_{2,\ell}$ as a subgraph.
\end{theorem}

\noindent
The rest of the paper is structured as follows. In Section~\ref{sec:leveldconstr} we provide a construction proving Theorem~\ref{thm:level}. Section~\ref{sec:mainproof} contains the proof of our main result, Theorem~\ref{thm:groundedclass}. Finally, in Section~\ref{sec:spiders} we prove Theorem~\ref{thm:linear_spiders}. 

\noindent
\paragraph*{Notation.} Most of our notation is standard. $G=(V,E)$ denotes a \emph{digraph} (directed graph) with the vertex set $V$ and edge set $E\subseteq\{(u,v)\in V\times V: u\neq v\}$. That is, we do not allow loops and multiple copies of the same edge, but we do allow two edges in opposite directions between a pair of vertices. $G$ is an \emph{oriented graph} if between any two vertices there is at most one edge. 

An edge $(u,v)$ is considered oriented from $u$ to $v$. The \emph{in- and out-neighbourhoods} of a vertex $v\in V$ are defined as $N^-(v)=\{u\in V: (u,v)\in E\}$ and $N^+(v)=\{u\in V: (v,u)\in E\}$, respectively. The \emph{in- and out-degrees} of $v$ are $\deg^-(v)=\card{N^-(v)}$ and $\deg^+(v)=\card{N^+(v)}$, respectively. The \emph{minimum out-degree} in $G$ is $\delta^+(G)=\min\{\deg^+(v): v\in V\}$.
For a vertex set $W\subseteq V$ we use $G[W]$ to denote the subgraph of $G$ induced on $W$. Furthermore, for a vertex $v\in V$ we denote $N^-_W(v)=N^-(v)\cap W$, $\deg^-_W(v)=\card{N^-_W(v)}$, and similarly for $N^+_W(v)$ and $\deg^+_W(v)$. 

A \emph{directed path} is an orientation of an undirected path which can be traversed from one end to the other following the orientations of the edges. A \emph{subdivision} of a digraph $G$ is obtained by replacing each edge of $G$ with a directed path (possibly of length $1$) in the same direction. When each edge is replaced by a path of length exactly $k$, we speak of a $(k-1)$-subdivision. An \emph{oriented tree} is an orientation of an undirected tree. An \emph{in-arborescence}/\emph{out-arborescence} is an oriented tree in which all edges are oriented towards/away from a designated root vertex. By $B^+_{k,\ell}$ we denote the $\ell$-branching out-arborescence of depth $k$, i.e. the complete $\ell$-ary tree of depth (distance from the root to the leaves) $k$ oriented away from the root.
An \emph{in-star}/\emph{out-star} is the orientation of an undirected star towards/away from its centre. We denote by $\smkl$ the $(k-1)$-subdivision of the in-star with $\ell$ leaves, and the directed paths from the leaves of $\smkl$ to its centre are referred to as the \emph{rays}.
\section{The level digraph construction}\label{sec:leveldconstr}

In this section we prove that every $\delta^+$-enforcible tree must be grounded. 

\begin{definition}
    A \emph{level digraph} $G_{k,d} = (V,E)$ is the digraph with the vertex set
\[
V = \{ v_{i,j} : 0 \le i \le k, \ 1 \le j \le d^{i+1} \},
\]
and the edge set
\[
E = \biggl( \bigcup_{i=0}^{k-1} E_{i,i+1} \biggr) \cup E_{k,0},
\]
where, for \(0 \le i < k\), 
\[
E_{i,i+1} = \Bigl\{ ( v_{i,j}, v_{i+1,(j-1)d + \ell}) : 1 \le j \le d^{i+1},\; 1 \le \ell \le d \Bigr\},
\]
and
\[
E_{k,0} = \Bigl\{ (v_{k,j},\, v_{0,\ell}) : 1 \le j \le d^{k+1},\; 1 \le \ell \le d \Bigr\}.
\]

\end{definition}

In other words, to construct $G_{k,d}$ we take $d$ disjoint copies of $B^+_{k,d}$ and add the edges from every leaf to every root.

\begin{figure}[htbp]
  \centering
  \includegraphics[scale=0.85]{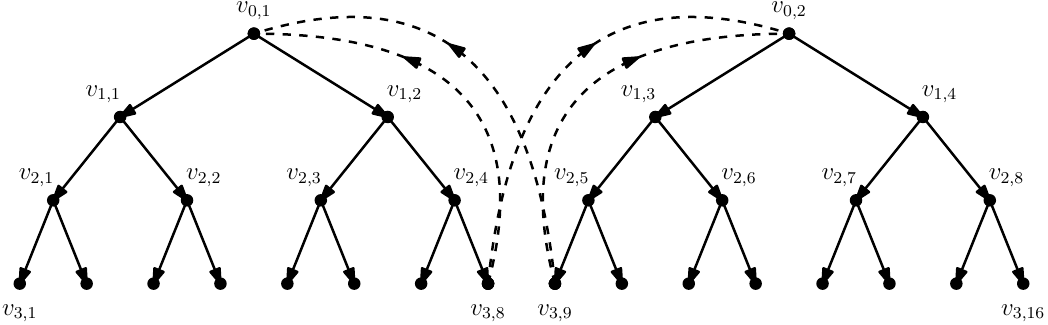}
  \caption{The level digraph \(G_{3,2}\).} 
\end{figure}

\begin{observation}\label{obs:level_graph}
    In $G_{k,d} = (V,E)$ for all $v \in V$ we have: 
    \begin{itemize}
        \item $ \deg^+(v) = d$. 
        \item \begin{flushleft}
        $\displaystyle
        \deg^-(v)=
        \begin{cases}
          d^{k+1}, & \text{if } v = v_{0,j} \text{ for  some } 1 \le j \le d, \\
          1, & \text{otherwise}.
        \end{cases}$
        \end{flushleft}
     \end{itemize}
\end{observation}

\begin{proof}[Proof of Theorem~\ref{thm:level}]    
       
    Let $T$ be a $\delta^+$-enforcible tree. We aim to show that $T$ is grounded, i.e., the height function $h_T$ is constant on $U = U(T)=\{v \in V(T) : \deg^-(v)\ge 2\}$. Let us assume that $U\neq \emptyset$, else $T$ is trivially grounded.
    Since $T$ is $\delta^+$-enforcible, there exists $d = d(T)$ such that every digraph $G$ with $\delta^+(G) \ge d$ contains $T$ as a subgraph. In particular, $G=G_{2t, d}=(V,E)$, where $t = \card{V(T)}$, contains $T$ as a subgraph, and let $\phi: T\rightarrow G$ be a witnessing embedding. For each $0 \le i \le 2t $, set $L_i = \{v_{i,j}  \in V\, \colon \, 1 \le j \le d^{i+1}\}$.  
    
    By Observation~\ref{obs:level_graph}, a vertex of $G$ has in-degree more than $1$ if and only if it belongs to $L_0$. Therefore $\phi(U)\subseteq L_0$, and in particular, $\phi(T)\cap L_0\neq \emptyset$. This implies $\phi(T)\cap L_t=\emptyset$, as $t=\card{V(T)}$ and the edges in $G$ are between $L_i$ and $L_{i+1}$ modulo $2t+1$.  So $\phi$ embeds $T$ in $G[V\setminus L_{t}]$. Note now that $G[V\setminus L_{t}]$ has a well-defined height function $\psi(v_{i,j})=i$ (equivalently, $G[V\setminus L_{t}]$ is homomorphic to a directed path), so $\psi \circ \phi$ is a height function of $T$ which is constant on $U$. Hence, $T$ is grounded.
\end{proof}
\section{\texorpdfstring{A new family of $\delta^+$-enforcible trees}{A new family of δ+-enforcible trees}}\label{sec:mainproof}

In this section we prove Theorem~\ref{thm:groundedclass}.
Let $T$ be a tree satisfying the assumptions of Theorem~\ref{thm:groundedclass}, and let $T^*$ be its subtree obtained by iteratively removing leaves of in-degree $1$ until none remain. Note that every vertex of $T^*$ has the same in-degree in $T^*$ as in $T$, in particular, $U(T^*) = U(T)$. Hence, $T^*$ satisfies the assumptions of Theorem~\ref{thm:groundedclass}. On the other hand, if $T^*$ is $\delta^+$-enforcible then so is $T$, by a greedy embedding. Thus it suffices to prove that $T^*$ is $\delta^+$-enforcible. Furthermore, if $\card{U(T)}=\card{U(T^*)} = 0$, then $T^*$ is a single vertex, which is trivially $\delta^+$-enforcible, and if $\card{U(T)}=\card{U(T^*)} = 1$, then $T^*$ is a subdivision of an in-star, which is also $\delta^+$-enforcible, as was shown in~\cite{aboulker1610subdivisions} (see Proposition~\ref{prop:abetal} below). Hence, relabelling, we may assume that $T$ does not have any leaf of in-degree $1$ and that $U = U(T)$ is of size at least $2$. 

Let $T'$ be the minimal subtree of $T$ containing $U$. By our assumption, $T'$ is an out-arborescence, and let $r$ be its root vertex. Since $|U|\geq 2$ and $T$ is grounded we have $r\notin U$. Moreover, since $T$ is grounded and $h_T$ must extend $h_{T'}$, the latter must be constant on the vertices of $U$. By the minimality of $T'$, $U$ must be the set of leaves of $T'$. 

Recall that $S^-_{k,\ell}$ is the $(k-1)$-subdivision of the in-star with $\ell$ leaves and $B^+_{k,\ell}$ is the $\ell$-branching out-arborescence of depth $k$.
Let $T(k,\ell)$ denote the oriented tree created from $B^+_{k,\ell}$ by identifying each leaf with the centre of a new copy of $S^-_{k,\ell}$.

\begin{lemma}\label{lem:inarb}
$T$ is a subgraph of $T(k,\ell)$ for some sufficiently large $k$ and $\ell$. 
\end{lemma}
\begin{proof}
By definition of $T'$, every vertex $v \in V(T) \setminus V(T')$ has $\deg^-(v) \leq 1$.
Moreover, it has $\deg^+(v) \leq 1$, as otherwise $T$ would contain a leaf of in-degree $1$.
Therefore, $T \setminus T'$ is a collection of vertex disjoint paths $\mathcal{F}$, each of which is directed towards $T'$.
Since each $v'\in V(T')\setminus\{r\}$ has an in-neighbour in $T'$, the paths in $\mathcal{F}$ can only connect to $T'$ via $r$ or a vertex of $U$.
Moreover, the former holds for at most one path $P_r\in \mathcal{F}$.

It follows that the subtree of $T$ induced by $V(T') \cup V(P_r)$ (with $P_r$ possibly empty) is an out-arborescence and can be embedded in $B^+_{k,\ell}$ for some sufficiently large $k$ and $\ell$ such that $U$, the set of leaves of $T'$, maps to the leaves of $B^+_{k,\ell}$.
The remaining paths in $\mathcal{F}$ can be naturally grouped into disjoint in-stars, centred in the vertices of $U$. Increasing the values $k,\ell$ if necessary, we obtain that $T$ is a subgraph of $T(k,\ell)$.
\end{proof}
\noindent
Consequently, in order to prove Theorem~\ref{thm:groundedclass}, it is enough to consider the trees $T = T(k,\ell)$. We now recall the following result from~\cite{aboulker1610subdivisions}, which will play a crucial rule in our proof.  

\begin{proposition}[\cite{aboulker1610subdivisions}]\label{prop:abetal}
There exists a function $f(k,\ell)$ such that every digraph $G$ with $\delta^+(G)\geq f(k,\ell)$ contains $S^-_{k,\ell}$ as a subgraph.     
\end{proposition}

\noindent
Note that the graph $T(k,1)$ is isomorphic to $S^-_{k,2}$, which is known to be $\delta^+$-enforcible by Proposition~\ref{prop:abetal}.
Therefore, we may assume that $\ell \geq 2$.
We shall prove the following quantitative form of Theorem~\ref{thm:groundedclass}.\footnote{For clarity of presentation, we do not attempt to optimise the bounds.}

\begin{theorem}\label{thm:supergraph_of_normalized}
    Let $k \geq 1, \ell \geq 2$. Any digraph $G$ with $\delta^+(G) \geq f(k,3k\ell^{k+1})+2k\ell^k$ contains $T(k,\ell)$ as a subgraph.
\end{theorem}
\noindent
In fact, we shall prove a more general theorem stating that if $\delta^+(G)$ is sufficiently large, then $G$ contains a copy of $B^+_{k,\ell}$ whose leaves satisfy any \emph{$\delta^+$-common} vertex property. 

\begin{definition}\label{def:common_property}
    Let $\Pro$ be a vertex property in digraphs.
    We say that $\Pro$ is \emph{$\delta^+$-common} if
    \begin{enumerate}[label=(\roman*)]
        \item there is $d = d(\Pro)$ such that every $G$ with $\delta^+(G) \geq d$ contains a vertex satisfying $\Pro$,
        \item $\Pro$ is anti-monotone.
        That is, if $H$ is a subgraph of $G$ and a vertex $v \in V(H)$ satisfies $\Pro$ in $H$, then $v$ satisfies $\Pro$ in $G$.
    \end{enumerate}     
\end{definition}
\noindent
A trivial example of a $\delta^+$-common property is ``$v$ is a vertex.''
A far less trivial (and important for us) $\delta^+$-common property is ``$v$ is the centre of a copy of $S^-_{k,\ell}$.''

\begin{theorem}\label{thm:outstar_with_leaves_property}
    Let $\Pro$ be a $\delta^+$-common property, $k \geq 1, \ell \geq 2$, and let $G$ be a digraph with $\delta^+(G) \geq d(\Pro)+2k\ell^k$. Then $G$ contains a copy $B$ of $B^+_{k,\ell}$ such that all leaves of $B$ satisfy $\Pro$ in $G$.
\end{theorem}
\noindent
Theorem~\ref{thm:supergraph_of_normalized} is a direct corollary of Theorem~\ref{thm:outstar_with_leaves_property}. 
\begin{proof}[Proof of Theorem~\ref{thm:supergraph_of_normalized}]
    Let $h=3k\ell^{k+1}$ and $\Pro$ be the property ``$v$ is the centre of a copy of $S^-_{k,h}$.'' By Proposition~\ref{prop:abetal}, $\Pro$ is $\delta^+$-common with $d(\Pro)=f(k,h)$. 
    Hence, $\delta^+(G) \geq d(\Pro)+2k\ell^k$ and we can apply Theorem~\ref{thm:outstar_with_leaves_property} to find in $G$ a copy $B$ of $B^+_{k,\ell}$, with the set of leaves $L$ of size $\ell^k$, such that each $w\in L$ satisfies $\Pro$ in $G$.
    
    So, each $w\in L$ is the centre of $S(w)$, a copy of $S^-_{k,h}$ in $G$.
    Now, take greedily from each $S(w)$ a subgraph $S'(w)$, a copy of $S^-_{k,\ell}$, such that all $S'(w)$ are disjoint from each other, and each $S'(w)$ is disjoint from $B$ save for $w$ --- together, $B$ and $\{S'(w):w\in L\}$ form a copy of $T(k,\ell)$ in $G$.
    This is possible since, when choosing $S'(w)$, the union of $B$ and all previously chosen $S'(w')$ contains at most 
    \[
        \card{V(B)} + \card{L} k\ell\leq 2\ell^k+k\ell^{k+1}
    \]
    vertices.
    Since each of them, except $w$, belongs to at most one ray of $S(w)$, there remain at least 
    \[
        h - 2\ell^k - k \ell^{k+1} \geq \ell
    \]
    rays of $S(w)$ that may be used to form $S'(w)$. 
\end{proof}

\begin{proof}[Proof of Theorem~\ref{thm:outstar_with_leaves_property}]
    Let $G=(V,E)$. We will inductively define sets $\Gamma(0), \dots, 
    \Gamma(k)\subseteq V$ that will guide our construction of the subgraph $B$.
    To do so, we define
    \begin{align*}
        \Gamma(0) &= \{ v \in V : \text{$v$ satisfies $\Pro$} \}, \\
        \Gamma(i) &= \{ v \in V : \deg^+_{\Gamma(i-1)}(v) \geq 2\ell^k \}\text{ for $1\leq i \leq k$.}
    \end{align*}
    \noindent
    The significance of the sets $\Gamma(i)$ is given by the following claim.
    
    \begin{claim}\label{claim:Gamma(k)_nonempty_implies_win}
    $\Gamma(k)\neq \emptyset$.
    \end{claim}
    Assuming Claim~\ref{claim:Gamma(k)_nonempty_implies_win}, we conclude the proof of Theorem~\ref{thm:outstar_with_leaves_property} as follows. Take some vertex $v\in \Gamma(k)$.
    We construct $B$ greedily from $v$ as the root.       
    For $i=0,\dots,k$ we will inductively construct $B_i$, a copy of $B^+_{i,\ell}$ in $G$, such that the leaves of $B_{i}$ (resp.~the single vertex of $B_0$) belong to $\Gamma(k-i)$.
    In particular, the leaves of $B_k$, a copy of $B^+_{k,\ell}$ in $G$, will satisfy $\Pro$ as they will belong to $\Gamma(0)$, so we can take $B=B_k$.
    
    To construct the trees $B_i$, set $B_0 = \{v\}$ and define $B_1$ to be an out-star with $v$ as the centre and $\ell$ of its out-neighbours in $\Gamma(k-1)$, which exist since $\deg^+_{\Gamma(k-1)}(v) \geq 2\ell^k > \ell$, as leaves.
    Now suppose we have constructed the tree $B_i$ for some $1\leq i \leq k-1$, and let  $L \subseteq \Gamma(k-i)$ be the set of its leaves.
    We have 
    \[
        \card{V(B_i)} = \sum_{j=0}^{i} \ell^j\leq 2\ell^{k-1},
    \]
    while $\deg^+_{\Gamma(k-i-1)}(w) \geq 2\ell^k$ for each $w\in L$. Hence, every $w\in L$ has at least $\ell^k$ neighbours in  $\Gamma(k-i-1)\setminus V(B_i)$. Since $\card{L} \leq \ell^{k-1}$, we can greedily choose for each $w\in L$ a set of $\ell$ out-neighbours in $\Gamma(k-i-1)\setminus V(B_i)$, such that the resulting sets are pairwise disjoint. These sets together with $B_i$ form a copy of $B_{i+1,\ell}^+$ in $G$, whose leaves belong to $\Gamma(k-i-1)$. We take this tree to be $B_{i+1}$.
\end{proof}
  
It remains to prove Claim~\ref{claim:Gamma(k)_nonempty_implies_win}. To this end, we partition the vertices of $G$ according to their presence in the sets $\Gamma(i)$.
That is, for a vertex $v\in V$ and an integer $0\leq i\leq k$, let $z_i(v)=1$ if $v\in \Gamma(i)$, and $z_i(v)=0$ otherwise, and let  $z(v) =(z_0(v),\dots,z_k(v))$. Conversely, for a vector $z =(z_0,\dots,z_k)\in \{0,1\}^{k+1}$ we set
\[
    V_z = \{ v \in V : z=z(v) \}
    .
\]
In this notation, Claim~\ref{claim:Gamma(k)_nonempty_implies_win} states that $V_z\neq \emptyset$ for some $z=(z_0,\dots,z_k)$ with $z_k = 1$.

Denote by $\prec$ the lexicographic ordering on $\{0,1\}^{\{0,\dots,k\}}$.
That is, for two vectors $z=(z_0\dots,z_k)$ and $z'=(z'_0,\dots,z'_k)$, we write $z \prec z'$ if there is an index $0\leq i \leq k$ such that $z_i = 0, z'_i = 1$ and $z_j = z'_j$ for all $j > i$.
We highlight the following property of the ordering.
\begin{observation}\label{obs:lex_order_property}
    If $z' \succ z$ and $z'_k = 0$, then there is an index $ i \leq k-1$ such that $z'_i = 1$ and $z_{i+1} = 0$.
\end{observation}
\noindent
We denote by $\vec{0}$ the all-zero vector $\vec{0} \in \{0,1\}^{k+1}$, and put $X=V_{\vec{0}}=V\setminus\bigcup_{i=0}^k\Gamma(i)$.

\begin{proof}[Proof of Claim~\ref{claim:Gamma(k)_nonempty_implies_win}]
    Suppose for a contradiction that $\Gamma_k=\emptyset$, or equivalently, 
    $V_{z'}=\emptyset$ for all $z'=(z'_0,\dots,z'_{k-1},1)$. Let $z =(z_0,\dots,z_k)$ be the $\prec$-smallest vector such that $V_z \not= \emptyset$; by the above assumption we have $z_k=0$. Let 
    \[
        I=\Big\{i \in \{0,\dots, k-1\}: z_{i+1}=0\Big\} \ \ \text{and} \ \
        W=\bigcup_{i\in I}\Gamma(i).
    \]
    We claim that
    \begin{align*}
        V =
        \begin{cases}
            W        &\text{ if } z \not= \vec{0},\\
            W \cup X &\text{ if } z = \vec{0}.
        \end{cases}
    \end{align*}
    Indeed, $\{V_{z'}:z'\in \{0,1\}^{k+1}\}$ is a partition of $V$, and for a vector $z'=(z'_0,\dots,z'_k)$, we have the following options:
    \begin{enumerate}[label=(\roman*)]
        \item if $z'_k = 1$, then $V_{z'} = \emptyset$ by assumption,
        \item if $z' \prec z$, we have $V_{z'} = \emptyset$ by the minimality of $z$,
        \item if $z' \succ z$ and $z'_k = 0$, then, by Observation~\ref{obs:lex_order_property}, for some $0\leq i \leq k-1$ we have $z'_i = 1$ and $z_{i+1} = 0$ . Therefore, $i\in I$ and $V_{z'} \subseteq \Gamma(i)\subseteq W$,
        \item lastly, if $z' = z \not= \vec{0}$, then, as $z_k=0$, for some $0\leq i \leq k-1$ we have $z_i = 1$ and $z_{i+1} = 0$. Therefore, $i\in I$ and $V_{z'} \subseteq \Gamma(i)\subseteq W$.
    \end{enumerate}
\noindent
    Let now $v\in V_z$ be an arbitrary vertex (recall that $V_z\neq \emptyset$ by definition of the vector $z$).
    For each $i\in I$ we have $v\notin \Gamma(i+1)$, as $z_{i+1}=0$. Hence,
    $\deg^+_{\Gamma(i)}(v)<2\ell^k$. Taking the sum over all $i\in I$ gives 
    \[
        \deg^+_W(v) < 2k\ell^k.
    \]
    So, if $z \not= \vec{0}$, then due to $V=W$, we get 
    \[
        2k\ell^k > \deg^+_W(v) = \deg^+(v) \geq \delta^+(G) \geq 2k\ell^k,
    \]
    a contradiction. 
    
    \noindent
    While if $z = \vec{0}$, we obtain 
    \[
        d_X^+(v) \geq \delta^+(G) - \deg^+_W(v) \geq \delta^+(G) - 2k\ell^k \geq d(\Pro).
    \]
    Since in this case $v$ was an arbitrary vertex of $V_z=V_{\vec{0}}=X$, it follows that the induced subgraph $G[X]$ satisfies $\delta^+(G[X]) \geq d(\Pro)$. Thus, there is a vertex $x \in X$ that satisfies the property $\Pro$ in $G[X]$.
    However, by the antimonotonicity of $\Pro$, $x$ must also satisfy $\Pro$ in $G$, meaning $x\in \Gamma(0)$. This is a contradiction with the fact that
    \[
        x \in X =V_{\vec{0}}\subseteq V\setminus \Gamma(0).
    \]
    Hence, our initial assumption that $\Gamma(k)=0$ was contradictory, and we must have $\Gamma(k)\neq 0$. 
\end{proof}
\section{Giant spiders exist}\label{sec:spiders}

In this section we prove Theorem~\ref{thm:linear_spiders}, which gives a linear bound on the minimum out-degree that guarantees a copy of $\smtwol$.

\begin{proof}[Proof of Theorem~\ref{thm:linear_spiders}]
    Let $\ell\geq 1$ be fixed and $G = (V, E)$ be a digraph with $\delta^+(G) \geq d$, where $d \ge (1+\sqrt{5})\ell$. By removing edges if necessary, we may assume that $\deg^+(v) = d$ for every vertex $v \in V$. We remark that this seemingly insignificant assumption is crucial for our argument. Our goal is to find a copy of $\smtwol$ in $G$. To this end, we partition $V$ into subsets $A$ and $B$, where 
    $$A = \{v \in V : \deg^-(v) \geq 2\ell\} \ \ \text{ and } \ \ B =\{v \in V : \deg^-(v) < 2\ell\},$$
    and note that $A\neq \emptyset$, since the average in-degree in $G$ is $d>2\ell$.

    If there exists a vertex $r \in A$ such that $\deg^-_A(r) \geq \ell$, we exhibit a copy of $\smtwol$ in $G$ as follows:
    First, we arbitrarily choose distinct vertices $a_1, a_2,\dots, a_\ell$ where $a_i\in N^-_A(r)$ and $1 \le i \le \ell$.
    Then we greedily select vertices $x_1, x_2,\dots, x_\ell \in V$ such that $x_i\in N^-_{V}(a_i)\setminus \{r, a_1,\dots,a_{i-1},a_{i+1}, \dots, a_\ell, x_1, \dots, x_{i-1}\}$.
    We will not run out of vertices since the vertex $a_i$ has at least $2\ell$ in-neighbours and, at the time of selecting the vertex $x_i$, at most $1 + (\ell-1) + (i-1) < 2\ell$ of them are already used in our construction.

    Thus, for the rest of the proof we may assume that $\deg^-_A(v) \leq \ell-1$ for every $v \in A$. For three sets $X,Y,Z \subseteq V$, not necessarily distinct, we denote by $X \rightarrow Y \rightarrow Z$ the set of all $2$-edge directed paths, with the vertex set $\{x,y,z\}$ for some $x \in X$, $y\in Y$, and $z \in Z$ and the edge set $\{(x,y),(y,z)\}$.

    \begin{figure}[htbp]
        \centering
        \includegraphics[scale=0.85]{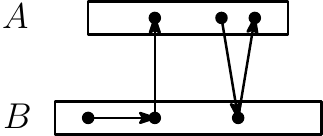}
        \caption{An example of paths in $V \rightarrow B \rightarrow A$, one starting in $B$ and the other in $A$.}
    \end{figure}
    
    Observe that
    \begin{align*}
        \card{V \rightarrow B \rightarrow A} &=  \card{V \rightarrow B \rightarrow V} - \card{V \rightarrow B \rightarrow B}\\
        &= \card{A \rightarrow B \rightarrow V} + \card{B \rightarrow B \rightarrow V} - \card{V \rightarrow B \rightarrow B}\\
        &= \card{A \rightarrow V \rightarrow V} - \card{A \rightarrow A \rightarrow V} + \card{B \rightarrow B \rightarrow V} - \card{V \rightarrow B \rightarrow B}, 
    \end{align*}
    and let us estimate the terms in the last expression.
    Because $\deg^+(v)=d$ for all $v$, $\deg^-_A(v) \leq \ell-1$ for every $v \in A$, and $\deg^-(v) \leq 2\ell-1$ for every $v \in B$, we obtain:
    \begin{align*}
        \card{A \rightarrow V \rightarrow V} &\geq \card{A} d(d-1), \\
        \card{A \rightarrow A \rightarrow V} &\leq e(G[A]) d \leq (\ell-1) d\card{A}, \\
        \card{B \rightarrow B \rightarrow V} &\geq e(G[B])(d-1), \\
        \card{V \rightarrow B \rightarrow B} &\leq e(G[B])(2\ell-1).
    \end{align*}
    Hence
    \begin{align*}
        \card{V \rightarrow B \rightarrow A} &\geq d(d-1)\card{A} - (\ell-1) d\card{A} + e(G[B])(d-1) - e(G[B])(2\ell-1) \\
            &\geq d(d - \ell )\card{A} + e(G[B])(d - 2\ell)\\
            &\geq d(d - \ell)\card{A}.
    \end{align*}
    By the pigeonhole principle, there exists a vertex $a \in A$ such that 
    \[
        \card{V \rightarrow B \rightarrow \{a\}} \geq d(d-\ell).
    \]

    Let $s$ be maximal such that there exists $S=\{b_1,\dots,b_s\}\subseteq N_B^-(a)$ and $Q=\{q_1,\dots,q_s\}\subseteq V\setminus (S\cup \{a\})$ with $(q_i,b_i)\in E$ for all $1\leq i\leq s$. Note that this gives a copy of $S^-_{2,s}$ in $G$. Since the in-degrees in $B$ are at most $2\ell$, we have 
    \[
        \card{V \rightarrow S \rightarrow \{a\}} \leq 2 \ell s,
    \]
    and similarly
    \[
        \card{V \rightarrow Q \cap B \rightarrow \{a\}} \leq 2 \ell s.
    \]
    Let $R \coloneqq B\setminus (S \cup Q)$. Consequently, we obtain
    \[
        \card{V \rightarrow R \rightarrow \{a\}} \geq  d(d-\ell) - 4\ell s.
    \]

    By maximality of $s$, all these paths have their first vertex in $S \cup Q$. 

    Now, fix a pair $(q_i,b_i)$. Set $X \coloneqq N_B^-(a) \cap N^+(q_i) \cap R$ and $Y \coloneqq N_B^-(a) \cap N^+(b_i) \cap R$. If there exist some distinct vertices $x \in X,y \in Y$, then we can extend $S $ and $Q$ by one element each (by adding $x,y$ to $S$ and moving $b_i$ to $Q$), contradicting the maximality of $s$. 
    
    Hence, no such pair of distinct vertices exists. Consequently, either $X=Y$ with $\card{X} = 1$, or $\card{X} = 0$ (or symmetrically $\card{Y}$ = 0). In the former case, there are at most two paths of the form $\{q_i,b_i\} \rightarrow R \rightarrow \{a\}$. In the latter case, we have $\card{N_B^-(a) \cap (N^+(q_i) \setminus \{b_i\})} \le d$ (or $\card{N_B^-(a) \cap N^+(b_i)} \le d$, respectively), since, by assumption, $\deg^+(v) = d$ for every vertex $v \in V$.
    Thus, we have at most $d$ paths of the form $\{q_i,b_i\} \rightarrow R \rightarrow \{a\}$.

    Therefore, there are at most $sd$ such paths starting at some $z \in S \cup Q$ and ending at $a$, which implies,
    \[
        sd \geq \card{S \cup Q\rightarrow R \rightarrow \{a\}} \geq {d(d-\ell) - 4\ell s},
    \]
    resulting in 
    \[
        s\geq \frac{d(d-\ell)}{d+4\ell}.
    \]
    The last expression is at least $\ell$ for $d \geq (1+\sqrt{5})\ell \approx 3.23 \cdot\ell $, since $(1+\sqrt{5})\ell$ is the positive root of the underlying quadratic equation. Thus, $G$ contains a copy of $\smtwol$.
\end{proof}
\section{Acknowledgement}

We would like to thank the organizers of the KAMAK 2024 workshop for the outstanding hospitality that facilitated the beginning of our work on this paper. The authors would also like to thank Grzegorz Gutowski for drawing attention to a mistake in an earlier draft and for fruitful discussions.

\printbibliography

\end{document}